\documentclass[letterpaper, 12pt]{article}

\usepackage{amsmath,amssymb,amsfonts}%
\usepackage{amsthm}%
\usepackage{mathrsfs}%

\newtheorem{corollary}{Corollary}
\newtheorem{lemma}{Lemma}

\newtheorem{definition}{Definition}%

\newcommand{\F}{\mathbb{F}}
\newcommand{\K}{\mathcal{K}}
\newcommand{\G}{\mathcal{G}}

\title{When Arcs Extend Uniquely: A Higher-Dimensional Generalization of Barlotti’s Result }
\author{Tim L. Alderson}

\begin{document}

\maketitle

\begin{abstract}
	In this short communication, we generalize a classical result of Barlotti concerning the unique extendability of arcs in the projective plane to higher-dimensional projective spaces. Specifically, we show that for integers \( k \ge 3 \), \( s \ge 0 \), and prime power \( q \), any \((n, k + s - 1)\)-arc in PG\((k - 1, q)\) of size \( n = (s+1)(q+1) + k - 3 \) admits a unique extension to a maximal arc, provided \( s + 2 \mid q \) and \( s < q - 2 \). This result extends the classical characterizations of maximal arcs in PG\((2,q)\) and connects naturally to the theory of A$^s$MDS codes. Our findings establish conditions under which linear codes of given dimension and Singleton defect can be uniquely extended to maximal-length projective codes.
\end{abstract}

\textit{MSC2010:} Primary: 94B65; Secondary: 94B25, 94B27

\textit{Keywords:} $A^{s}$MDS codes; $(n,r)$-arcs; maximal arcs; code extensions; singleton bound; singleton defect

\section*{Introduction}

The study of arcs in finite projective spaces has long stood at the intersection of finite geometry and coding theory. In the projective plane PG\((2,q)\), an arc is a set of points with the property that no three are collinear, and such a set of size \( q + 1 \) is called an oval. A classical result of Segre  [\cite{MR0075608} 1955] shows that every arc of size \( q+1 \) in PG\((2,q) \) lies on a conic, and that such arcs admit an unique extension to an arc of size \( q + 2 \) (a hyperoval) if and only if \( q \) is even. This phenomenon of unique extendability, tightly constrained by the parity of \( q \), initiated a rich line of inquiry into extremal configurations in projective geometries.

In 1956, Barlotti \cite{MR0083141} extended Segre’s result by proving that \((n,r)\)-arcs in PG\((2,q)\) of size \( n = (s+1)(q+1) \), where \( s = r - 2 \), also exhibit a unique extension property, again contingent on the divisibility condition \( s+2 \mid q \). His result laid the groundwork for a connection to extremal coding theory, as such arcs correspond to $3$-dimensional linear codes with  Singleton defect $s$.

In this note, we build on this foundation by generalizing Barlotti’s unique extension result to projective spaces of arbitrary dimension. Specifically, we prove that for \( k \ge 3 \), any \((n, k+s-1)\)-arc in PG\((k - 1, q)\) of length \( n = (s+1)(q+1) + k - 3 \) admits a unique extension to a maximal arc, provided \( s + 2 \mid q \) and \( s < q - 2 \). This result bridges a gap, showing that unique extendability previously known in the planar case persists in higher-dimensional projective geometries under analogous conditions.

Geometrically, the proof relies on a quotient construction, exploiting the behaviour of tangents and secants under projection. 


\section*{Projective Geometry and Quotient Geometries}

The projective space of dimension \( k-1 \) over \( \F_q \) is denoted by PG\((k - 1, q)\). A \emph{flat} in PG\((k - 1, q)\) is a projective subspace; specifically, a \( d \)-flat refers to a subspace of projective dimension \( d \). Flats of codimension 1 are called \emph{hyperplanes}.

Given a flat \( \Lambda \) of PG\((k - 1, q)\), the quotient geometry with respect to \( \Lambda \), is (isomorphic to)   PG\((k - \dim(\Lambda) - 2, q)\). It is defined by associating  points with the \((\dim(\Lambda)+1)\)-flats in PG\((k - 1, q)\) that contain \( \Lambda \), and higher-dimensional flats with those containing \( \Lambda \).
Incidence in the quotient geometry is inherited from containment in the original projective space: a point (i.e., \((\dim(\Lambda)+1)\)-flat) lies in a flat of the quotient if the corresponding flat in PG\((k - 1, q)\) is contained in the corresponding larger flat containing \( \Lambda \).

\section*{Arcs and Multi-Arcs in Projective Space}

An \emph{\((n, r)\)-arc} \( \K \) in PG\((k - 1, q)\), where \( r \ge k - 1 \), is a set of \( n>r \) points such that:
\begin{itemize}
	\item every hyperplane intersects \( \K \) in at most \( r \) points, and
	\item there exists at least one hyperplane that meets \( \K \) in exactly \( r \) points.
\end{itemize}

If \( \K \) is a multiset — that is, some points occur with multiplicity greater than one — it is called a \emph{multi-arc}.

An arc is called \emph{complete} if it cannot be extended by adding another point without violating the arc condition. That is, no additional point can be included to form an \((n+1, r)\)-arc in PG\((k - 1, q)\).

Since each hyperplane is a \( (k - 2) \)-flat, the arc condition requires that \( r \ge k - 1 \). Thus, we may write \( r = k - 1 + s \), where \( s \ge 0 \) is called the \emph{defect parameter}.

Relative to an $(n, k-1+s)$-arc $\mathcal{K}$  in $\mathrm{PG}(k-1,q)$, hyperplanes are classified as \emph{secant}, \emph{tangent}, or \emph{external} according to whether they contain $k-1+s$, a positive number fewer than $k-1+s$, or zero points of $\mathcal{K}$, respectively.

\section*{Maximal Arcs and the Function \( m^s(k, q) \)}

For a given defect parameter \( s \ge 0 \), a simple counting argument  shows that an \((n, k + s - 1)\)-arc in PG\((k - 1, q)\) satisfies the inequality:
 
\begin{equation} \label{eqn: maximal arc bound}
	n \le (s+1)(q+1) + k - 2.
\end{equation}
 
An arc that meets this bound with equality is called a \emph{maximal arc}. These arcs are necessarily complete.

We define \( m^s(k, q) \) to denote the maximum size of a complete \((n, k + s - 1)\)-arc in PG\((k - 1, q)\). That is,
\[
m^s(k, q) := \max \{ n \mid \text{there exists a complete } (n, k + s - 1)\text{-arc in PG}(k - 1, q) \}.
\]
In particular, the inequality (\ref{eqn: maximal arc bound}) provides $m^s(k,q)\le (s+1)(q+1)+k-2$. 
As we shall see, the function $m^s(k,q)$ captures the largest possible length of a linear $q$-nary A$^s$MDS code of dimension $k$.

\section*{Maximal Arcs in the Plane}

In the special case of the projective plane PG\((2, q)\), several classical results establish precise bounds for maximal arcs, particularly in relation to the divisibility of \( q \) by \( s + 2 \). The results of Barlotti (1956) \cite{MR0083141}, when paired with the later developments of Ball et. al. \cite{MR1466573}, give bounds on $(n,r)$-arcs in PG$(2,q)$ which provide the first two items in the following Lemma. The third item is due to the construction of Denniston \cite{MR0239991}.

\begin{lemma}[\cite{MR0083141},\cite{MR1466573}, \cite{MR0239991}]  \label{lem: Barlotti 3d}  \leavevmode
	\begin{enumerate}
		\item If \( 0 < s < q - 2 \), and \( (s+2,q) \ne (2^e,2^h) \), then \( m^s(3,q) \le (s+1)(q+1)-1 \).
		\item If \( 0 < s < q - 2 \), \( (s+2) \mid q \), and \( m^s(3,q) \ge (s+1)(q+1) \), then \( m^s(3,q) = (s+1)(q+1)+1 \).
		\item If \( 0 < s \le q - 2 \), and \( (s+2,q) = (2^e,2^h) \), then \( m^s(3,q) = (s+1)(q+1)+1 \).
	\end{enumerate}
\end{lemma}

Barlotti's foundational result establishes the unique extendability of certain arcs in the plane:

\begin{lemma}[\cite{MR0083141}]\label{lem Barlotti Extension k = 3}
	If \( \K \) is an \(((s+1)(q+1), s+2)\)-arc in PG\((2, q)\), with \( 0 < s < q - 2 \), and \( (s+2) \mid q \), then there exists a unique point \( P \) not incident with any secant of \( \K \), and consequently incident with all tangents to \( \K \). Hence, \( \K \) admits a unique extension to a maximal \(((s+1)(q+1)+1, s+2)\)-arc.
\end{lemma}

\section*{Maximal Arcs in Higher Dimensions}

We now generalize the results on unique extendability of planar arcs to projective spaces of higher dimension. We begin by establishing a necessary condition on the structure of long arcs:

\begin{lemma}\label{lem: long implies projective and cap}
	Let \( \K \) be an \((n, k+s-1)\)-(multi-)arc in PG\((k-1, q)\), with \( k \ge 3 \). If \( n > s(q+1) + k - 1 \), then each $(k-3)$-flat intersects $\K$ in at most $k-2$ points. Consequently,  \( \K \) must be a set (i.e., not a multi-arc).
\end{lemma}

\begin{proof}
	 If some \((k-3)\)-flat \( \Lambda \subset \text{PG}(k-1,q) \) contains $k-1$ points of $\K$, then by considering the $q+1$ hyperplanes through $\Lambda$ we obtain $n\le s(q+1)+k-1$.  
The result follows.
 \end{proof}

Using this observation, we obtain the following generalization of Barlotti's unique extension lemma:

\begin{lemma}\label{lem: main result}
	Let \( \K \) be an \((n, k + s - 1)\)-arc in PG\((k - 1, q)\), with \( k \ge 3 \), \( 0 < s < q - 2 \), and \( (s+2) \mid q \). If \( n = (s+1)(q+1) + k - 3 \), then there exists a unique point \( X \in \text{PG}(k - 1, q) \) that is incident with all tangents to \( \K \), and not incident with any secant of \( \K \). Consequently, \( \K \) is not complete and admits a unique extension to a maximal \((n + 1, k + s - 1)\)-arc.
\end{lemma}

\begin{proof}
	By Lemma~\ref{lem: long implies projective and cap}, the arc \( \K \) must be a set (i.e., not a multi-arc). We proceed by induction on the projective dimension \( k - 1 \).
	
	\smallskip
	\noindent \text{Base case \( k = 3 \):} This is precisely Lemma~\ref{lem Barlotti Extension k = 3}, which establishes the existence of a unique point not lying on any secant of \( \K \) and incident with all tangents.
	
	\smallskip
	Assume the result holds for PG\((k-2, q)\). Let \( \K = \{P_1, P_2, \ldots, P_n\} \subset \Pi = \text{PG}(k-1, q) \), with \( n = (s+1)(q+1) + k - 3 \), and \( k \ge 4 \). For each point \( P_i \in \K \), consider the quotient geometry \( \Pi^*_i = \Pi / P_i \cong \text{PG}(k - 2, q) \), and let \( \K_i^* \) be the image of \( \K \setminus \{P_i\} \) under this quotient.
	
	By construction, \( \K_i^* \) is an arc of size \( (s+1)(q+1) + k - 4 \) in PG\((k - 2, q)\), satisfying the hypotheses of the inductive step. Thus, there exists a unique line \( \ell_i \subset \Pi \) through \( P_i \), lying entirely in the tangents to \( \K \) at \( P_i \), and meeting every secant through \( P_i \) only at \( P_i \) itself.
	
	\smallskip
	\noindent \textit{Claim:} The lines \( \ell_i \) are pairwise distinct. Suppose \( \ell_i = \ell_j \) for some \( i \ne j \). In the quotient at \( P_i \), the image \( P_j^* \in \K_i^* \) corresponding to $\ell_j$ is incident with no secants. Thus $\K_i^*$ admits an extension to a multi-arc, violating Lemma~\ref{lem: long implies projective and cap}. 
	
	\smallskip
	\noindent \text{ \( k = 4 \):} Counting shows that for $i\ne j$, the points $P_i,P_j$ are incident with an unique tangent of $\K$, denoted $\pi_{ij}$, so $\ell_i$ and $\ell_j$ meet (in $\pi_{ij}$). Take any $P_a \notin \pi_{12}$. The line $\ell_a$ similarly meets both $\ell_1$ and $\ell_2$, and must meet $\pi_{12}$ in an unique point, say $X$.  Thus, each $\ell_i$ in $\pi_{12}$ meets $\ell_a$ in $X$, and each $\ell_i\notin \pi_{1,2}$ meets $\ell_1$ and $\ell_2$ in $X$, so all \( \ell_i \) are concurrent at \( X \).
	
	\smallskip
	\noindent \text{ \( k \ge 5 \):} With appeal to Lemma \ref{lem: long implies projective and cap}, for any pair \( P_i, P_j \in \K \), the quotient at \( \langle P_i, P_j \rangle \) yields a maximal arc in PG\((k-3, q)\). The inductive hypothesis implies that  there exists an unique plane $\pi_{ij}$ through $\langle P_i,P_j\rangle $ which is contained in each tangent of $\K$ through $\langle P_i,P_j\rangle $, and which meets each secant of $\K$ through $\langle P_i,P_j\rangle $ precisely in $\langle P_i,P_j\rangle $.  Note that by definition, $\ell_i,\ell_j \subseteq \pi_{ij}$ for all $i,j$.

	We now claim that the set \( \{\ell_i\} \) is copunctal.  Indeed, consider $\ell_1, \ell_2$, and $\ell_3$.  Since $\ell_1,\ell_2 \subseteq \pi_{12}$, they meet at a point, $X$. If $P_3\notin \pi_{12}$ then $\ell_1=\pi_{13}\cap\pi_{12}$, so $\ell_3\cap \pi_{12}=Y\in \ell_1$.  Likewise, by considering $\pi_{23}$ we obtain $\ell_3\cap \pi_{12}=Y\in \ell_2$, giving $Y=X$.  Thus $X\in \ell_1,\ell_2,\ell_3$. Hence, all \( \ell_i \) pass through the common point \( X \), which is incident with all tangents and avoids all secants of \( \K \), completing the proof.
\end{proof}

\begin{corollary}\label{cor: Barlotti bound on m^s(k,q)}
	If \( 0 < s < q - 2 \), \( (s+2) \mid q \),  $k\ge 3$ and \( m^s(3, q) \ge (s+1)(q+1) +k-3\), then  
	\[
	m^s(k, q) = (s+1)(q+1) + k - 2.
	\]
\end{corollary}

\section*{Interpretation in Terms of Linear Codes}

A linear code \( C \) of length \( n \), dimension \( k \), and minimum distance \( d \), denoted \([n, k, d]_q\), is a \( k \)-dimensional subspace of \( \F_q^n \). The dual code is defined as
\[
C^\perp = \{ y \in \F_q^n \mid y \cdot x = 0 \text{ for all } x \in C \},
\]
and has parameters \([n, n-k, d^\perp]_q\). A code is called \emph{non-degenerate} if no coordinate is identically zero on all codewords.

\subsubsection*{Projective Systems and A$^s$MDS Codes}

A linear code can be represented by a generator matrix whose columns form a multiset \( \G \) of points in PG\((k - 1, q)\). This representation, known as a \emph{projective system}, provides a geometric lens through which coding properties can be studied.

\begin{definition}
	A \emph{projective \([n, k, d]_q\) system} is a multiset \( \G \subseteq \text{PG}(k - 1, q) \) of size \( n \), such that:
	\[
	n - d = \max \{ |\G \cap H| \mid H \text{ is a hyperplane in PG}(k - 1, q) \}.
	\]
\end{definition}

A code is called \emph{projective} if \( \G \) is a set (i.e., all multiplicities equal 1).
The dual distance of a projective system \( \G \) is given by:

\begin{equation}\label{eqn: dual dist from proj system}
	d^\perp = \min \{ |\mathcal{Q}| \mid \mathcal{Q} \subset \G,~ |\mathcal{Q}| - \dim \langle \mathcal{Q} \rangle = 1 \},
\end{equation}

where \( \langle \mathcal{Q} \rangle \) denotes the linear span.


The \emph{Singleton bound} for a linear \([n, k, d]_q\) code $\G$ states that
$n - d \ge k - 1$.
We define the \emph{Singleton defect} as \( S(\G) = n - k + 1 - d \), and say that \( \G \) is  \emph{MDS} (Maximum Distance Separable) if \( S(\G) = 0 \), and is \emph{A$^s$MDS} (Almost MDS with defect $s$) if \( S(\G) = s\ge 0 \).

\subsubsection*{Length-Maximal Codes}

With reference to (\ref{eqn: dual dist from proj system}) we may restate Lemma~\ref{lem: long implies projective and cap} in coding terms:

\begin{corollary}\label{cor: long is projective}
	Let \( \G \) be an A$^s$MDS \([n, k, d]_q\) code. If \( n > s(q+1) + k - 1 \), then \( d^\perp \ge k \), and in particular \( \G \) is projective.
\end{corollary}

From Lemma~\ref{lem: Barlotti 3d} and the previous corollary, we deduce a universal bound:

\begin{corollary}\label{cor: length-maximal bound}
	Let \( \G \) be an A$^s$MDS \([n, k, d]_q\) code with \( k \ge 3 \). Then:
	\[
	n \le (s+1)(q+1) + k - 2,
	\]
\end{corollary}

A code meeting the bound in Corollary \ref{cor: length-maximal bound} is said to be  \textit{length-maximal}.


Applying Lemma~\ref{lem: main result}, we obtain the following:

\begin{corollary}
	Let \( \G \) be an A$^s$MDS \([n, k, d]_q\) code in PG\((k - 1, q)\), $k\ge 3$, with \( n = (s+1)(q+1) + k - 3 \), \( 0 < s < q - 2 \), and \( (s+2) \mid q \). Then \( \G \) can be uniquely extended to a length-maximal A$^s$MDS code.
\end{corollary}

\section*{Conclusion}

In this note we have extended the classical result on unique extendability of arcs in the projective plane to the higher‐dimensional setting of \(\mathrm{PG}(k-1,q)\) for \(k\ge3\).  In particular, under the divisibility condition \((s+2)\mid q\) and the length condition 
\[
n = (s+1)(q+1) + k - 3,
\]
we have shown that any \((n, k + s - 1)\)-arc admits an  unique  extension to a maximal arc (and hence an unique augmentation of the corresponding A\(^s\)MDS code).  

Although much is now understood in the planar case (i.e., \(k=3\)), the generalisation to higher dimension provided here opens the question:\textit{ Are there  other  cases in which unique extendability still holds in dimension \(k>3\)}?  
Improvements beyond the present conditions (either relaxing \( (s+2)\mid q \) or increasing the length threshold) would be especially welcome.

\medskip  
\textbf{Acknowledgements.}   The author acknowledges the support of the Natural Sciences and Engineering Research Council of Canada (NSERC), [funding reference number 2019-04103]\\
Cette recherche a \'{e}t\'{e} financ\'{e}e par le Conseil de recherches en sciences naturelles et en g\'{e}nie du Canada (CRSNG), [numéro de r\'{e}f\'{e}rence 2019-04103].

\bibliographystyle{plain}
\bibliography{../../../../my-texmf/bibtex/bib/main} 

\begin{thebibliography}{1}

\bibitem{MR1466573}
Simeon Ball, Aart Blokhuis, and Francesco Mazzocca.
\newblock Maximal arcs in {D}esarguesian planes of odd order do not exist.
\newblock {\em Combinatorica}, 17(1):31--41, 1997.

\bibitem{MR0083141}
Adriano Barlotti.
\newblock Sui {$\{k;n\}$}-archi di un piano lineare finito.
\newblock {\em Boll. Un. Mat. Ital. (3)}, 11:553--556, 1956.

\bibitem{MR0239991}
R.~H.~F. Denniston.
\newblock Some maximal arcs in finite projective planes.
\newblock {\em J. Combinatorial Theory}, 6:317--319, 1969.

\bibitem{MR0075608}
Beniamino Segre.
\newblock Curve razionali normali e {$k$}-archi negli spazi finiti.
\newblock {\em Ann. Mat. Pura Appl. (4)}, 39:357--379, 1955.

\end{thebibliography}

\end{document}